\documentclass[11pt,reqno]{amsart}

\usepackage{amscd,amssymb,amsmath,amsthm}
\usepackage{graphicx}
\usepackage{color}
\usepackage{cite}
\usepackage[shortlabels]{enumitem}

\newtheorem{thm}{Theorem}
\newtheorem{defn}{Definition}
\newtheorem{lemma}{Lemma}
\newtheorem{pro}{Proposition}
\newtheorem{rk}{Remark}

\usepackage{todonotes}

\numberwithin{equation}{section} 

\def\s{\sigma}

\def\s{\sigma}

\def\O{\Omega}

\def\s{\sigma}

\def\O{\Omega}

\def\m{\mu}
\def\n{\nu}
\def\L{\Lambda}

\def\Z{\mathbb{Z}}

\def \L {\Lambda}

\begin{document}
\title[Gradient Gibbs measures for the SOS model]{
Gradient Gibbs measures
for the SOS model with countable values on a Cayley tree
}

\author{F. Henning, C. K\"ulske,   A. Le Ny, U. A. Rozikov}

\address{F. Henning \\ Fakult\"at f\"ur Mathematik,
Ruhr-University of Bochum, Postfach 102148,\,
44721, Bochum,
Germany.}
\email{Florian.Henning@ruhr-uni-bochum.de}

\address{C.\ K\"ulske\\ Fakult\"at f\"ur Mathematik,
Ruhr-University of Bochum, Postfach 102148,\,
44721, Bochum,
Germany.}
\email {Christof.Kuelske@ruhr-uni-bochum.de}

\address{A. \ Le Ny \\ Universit\'e Paris-Est, Laboratoire d'Analyse et de Math\'ematiques Appliqu\'ees, LAMA UMR CNRS 8050, UPEC, 91 Avenue du G\'en\'eral de Gaulle, 94010 Cr\'eteil cedex, France.}
\email {arnaud.le-ny@u-pec.fr}

\address{U.\ A.\ Rozikov\\ Institute of mathematics,
81, Mirzo Ulug'bek str., 100125, Tashkent, Uzbekistan.}
\email {rozikovu@yandex.ru}

\begin{abstract} We consider an SOS (solid-on-solid) model,
with spin values from the set of all integers,  
on a Cayley tree of order $k\geq 2$ and are interested
in translation-invariant gradient Gibbs measures (GGMs) of the model. 
Such a measure corresponds to a boundary law (a function defined on vertices of the Cayley tree)
satisfying a functional equation.
In the ferromagnetic SOS case on the binary tree we find up to
five solutions to a class of $4$-periodic boundary law equations (in particular, some two periodic ones).
We show that these boundary laws define up to four distinct GGMs.
Moreover, we construct some $3$-periodic boundary laws on the Cayley tree of arbitrary order $k\geq 2$, which 
define GGMs different from the $4$-periodic ones.

\end{abstract}
\maketitle

{\bf Mathematics Subject Classifications (2010).} 82B26 (primary);
60K35 (secondary)

{\bf{Key words.}} {\em SOS model, Cayley tree,
Gibbs measure, tree-indexed Markov chain, gradient Gibbs measures, boundary law}. 

\section{Introduction}
We consider models where an infinite-volume spin-configuration $\omega$ is a function from the vertices of the tree Cayley
 to the local configuration space $E \subseteq \Z$.

A solid-on-solid (SOS) model is a spin system with spins taking values in (a subset of) the
integers, and formal Hamiltonian
$$
 H(\sigma)=-J\sum_{\langle x,y\rangle}
|\omega(x)-\omega(y)|,
$$
where $J\in \mathbb{R}$ is a coupling constant.
 As usual,
$\langle x,y\rangle$ denotes a pair of nearest neighbour vertices.

For the local configuration space we consider in the present paper the full set  $E:=\mathbb Z$.
The model can be considered as a generalization
of the Ising model, which corresponds to $E=\{-1,1\}$, or a less symmetric variant of the Potts model 
with non-compact state space.
SOS-models on the cubic lattice were analyzed in \cite{Maz} where an analogue of
the so-called Dinaburg--Mazel--Sinai theory was developed.
Besides interesting phase transitions in these models, the
attention to them is motivated by applications, in particular
in the theory of communication networks; see, e.g., \cite{Kel}, \cite{Ra}.
SOS models with $E=\Z$ have also been used as simplified discrete
interface models which should approximate the behaviour of a Dobrushin-state
in an Ising model when the underlying graph is $\Z^d$, and $d\geq 2$.
There is the issue of possible
non-existence of any Gibbs measure in the case of such unbounded spins, 
in particular in the additional presence of {\em  disorder} (see \cite{BK} and \cite{BK1}). 
In this paper we show that on the Cayley 
tree there are several translation invariant {\em gradient Gibbs measures}.
For more background on Gradient Gibbs measures 
on the lattice, also in the case of 
real valued state space, we refer to 
 \cite{FS97}, \cite{BiKo}, \cite{EnKu08}, \cite{CoKu12}, \cite{CoKu15} and \cite{BEvE}.

Compared to the Potts model, the $m$-state SOS model has
less symmetry: The full symmetry of the Hamiltonian
under joint permutation of the spin values is reduced to the mirror symmetry,
which is the invariance of the model under the map $\omega_i \mapsto m- \omega_i$ on the local spin space.
Therefore one expects a more diverse structure
of phases. 

To the best of our knowledge, the first paper devoted to the SOS model on the Cayley tree is \cite{Ro12}.
In \cite{Ro12} the case of  arbitrary $m\geq 1$ is treated
and
a vector-valued functional equation for possible boundary laws of the model is obtained.
Recall that each solution to this functional equation determines a splitting Gibbs measure (SGM),
in other words a tree-indexed Markov chain which is also a Gibbs measure. 
Such measures can be obtained by propagating
spin values along the edges of the tree, from {\em any site singled out to be the root} to the outside,
with a transition matrix depending on initial
Hamiltonian and the boundary law solution.
In particular the homogeneous (site-independent) boundary laws then
define translation-invariant (TI) SGMs. For a recent investigation of the influence of
weakly non-local perturbations in the interaction to the structure of Gibbs measures, 
see \cite{BEvE} in the context of the Ising model. 

Also the symmetry (or absence of symmetry) of the Gibbs measures under spin reflection 
is seen in terms of the corresponding boundary law. For SOS models some
TISGMs which are symmetric have already been studied 
in the particular case $m=2$  in \cite{Ro12}, and   $m=3$  in   \cite{Ro13}. In \cite{KR1}, for $m=2$, 
a detailed description of TISGMs (symmetric and non-symmetric ones) is given: it is shown the uniqueness in the case 
of antiferromagnetic interactions, and existence of up to seven
TISGMs in the case of ferromagnetic interactions.  See also \cite{Ro} for more details about SOS models on trees.

 In the situation of an unbounded local spin space the normalisability condition given in \cite{Z1} (which is needed to construct a SGM, in other words  a tree indexed Markov chain,  from a given boundary law solution) is {\em not} automatically satisfied anymore.
In this paper we are interested in the class of (spatially homogeneous/ tree-automorphism invariant) height-periodic boundary laws to tree-automorphism invariant potentials whose elements violate this normalisability condition. Here, a spatially homogeneous height-periodic boundary law with period $q$ is a $q$-periodic function on the local state space $\mathbb{Z}$. Although the procedure of constructing a Gibbs measures from boundary laws described in \cite{Z1} can not be applied to elements of that class, we are still able to assign a translational invariant \textit{gradient Gibbs measure} (GGM) on the space of gradient configurations to each such spatially 
homogeneous height-periodic boundary law, compare \cite{KS}. This motivates the study of spatially homogeneous height-periodic boundary laws as useful finite-dimensional objects which are are easier to handle than the non-periodic ones required to fulfill the normalisability condition.
Gradient Gibbs measures describe height differences, Gibbs measures describe 
absolute heights. Each Gibbs measure defines a gradient Gibbs measures, but the converse 
in not true, which is a phenomenon that is well-known from the lattice.
Some more explanation will be given in the following sections.  The main goal of this paper then consists in the description of a class of boundary solutions which have periods of 
$2$, $3$ and $4$ with respect to shift in the height direction on the local state space $\Z$, 
and their associated GGMs.

The paper is organized as follows. In Section 2 we first present the preliminaries of our model. Section 3 then contains a summary on the notion of GGMs on trees and their construction from homogeneous periodic boundary laws. For further details see \cite{KS}.
The main part, section 4,  is devoted to the description of a set of homogeneous $2$, $3$ and $4$-periodic boundary laws. Solving the associated boundary law equations for the $2$-periodic and the $4$-periodic case on the binary tree we prove that depending on the system parameters this set contains one up to five elements, yet the number of distinct GGMs assigned to them will turn out to be at most four. In the last subsection we construct GGMs for 3-periodic boundary laws on the $k$-regular tree for arbitrary $k \geq 2$.

\section{Preliminaries}

{\it Cayley tree.} The Cayley tree $\Gamma^k$ of order $ k\geq 1 $ (or $k$-\textit{regular tree})
is an infinite tree, i.e. a locally finite connected graph without cycles, such that
exactly $k+1$ edges originate from each vertex. Let $\Gamma^k=(V,
L)$ where $V$ is the set of vertices and  $L$ the set of edges.
Two vertices $x,y \in V$ are called {\it nearest neighbours} if
there exists an edge $l \in L$ connecting them. We will use the
notation $l=\langle x,y\rangle$. A collection of nearest neighbour
pairs $\langle x,x_1\rangle, \langle x_1,x_2\rangle,...,\langle
x_{d-1},y\rangle$ is called a {\it path} from $x$ to $y$. The
distance $d(x,y)$ on the Cayley tree is the number of edges of the
shortest path from $x$ to $y$.

Furthermore, for any $\Lambda \subset V$ we define its outer boundary as
 \begin{equation*}
\partial \Lambda := \{ x \notin \Lambda : d(x,y) = 1 \mbox{ for some } y \in \Lambda\}.
 \end{equation*}

{\it SOS model.} We consider a model where the spin takes values in
the set of all integer numbers $\mathbb Z:=\{\dots, -1,0,1,\dots
\}$, and is assigned to the vertices of the tree. A \textit{(height) configuration}
$\omega$ on $V$ is then defined as a function $x\in V\mapsto\omega_x\in\mathbb Z$; the set of all height configurations is $\Omega:=\mathbb Z^V$. Take the power set $2^\mathbb{Z}$  as measurable structure on $\mathbb{Z}$ and then endow $\Omega$ with the product $\sigma$-algebra $\mathcal{F}:=\sigma\{\omega_i \mid i \in V\}$ where $\omega_i: \Omega \rightarrow \mathbb{Z}$ denotes the projection on the $i$th coordinate. We also sometimes consider more general finite subsets $\Lambda$ of the tree and we write $\mathcal{S}$ for the set of all those finite subtrees. \\
Recall here that the (formal) Hamiltonian of the SOS model is
\begin{equation}\label{nu1}
 H(\sigma)=-J\sum_{\langle x,y\rangle\in L}
|\omega_x-\omega_y|,
\end{equation}
where $J \in \mathbb R$ is a constant which we will set to $1$ (incorporated in the inverse temperature $\beta$) in the following. As defined above, $\langle
x,y\rangle$ denotes nearest neighbour vertices.

Note that the above Hamiltonian depends only on the height difference between neighbouring vertices but not on absolute heights (it is given by a \textit{gradient interaction potential} in the terminology of \cite{KS}). This suggests reducing complexity of the configuration space by considering gradient configurations instead of height configurations as it will be explained in the following section.

 \section{Gradient Gibbs measures and an infinite system of functional equations}
\textit{Gradient configurations}: Let the Cayley tree be called $\Gamma^k$. We may induce an orientation on $\Gamma^k$ relative to an arbitrary site 
$\rho$ (which we may call the root) by calling an edge $\langle x,y \rangle$ \textit{oriented} iff it points away from the $\rho$. More precisely, the set of oriented edges is defined by
\begin{equation*}
 \vec L:=\vec{L_\rho}:= \{\langle x,y \rangle \in L \; : \; d(\rho,y)=d(\rho,x)+1\}.
\end{equation*} Note that the oriented graph $(V,\vec{L})$ also possesses all tree-properties, namely connectedness and absence of loops. \\For any height configuration $\omega = (\omega(x))_{x \in V} \in \mathbb{Z}^V$ and $b = \langle x,y \rangle \in \vec L$ the \textit{height difference} along the edge $b$ is given by $\nabla \omega_b = \omega_y - \omega_x$ and we also call $\nabla \omega$ the \textit{gradient field} of $\omega$. The gradient spin variables are now defined by $\eta_{\langle x,y \rangle} = \omega_y - \omega_x$ for each $\langle x,y \rangle \in \vec L$. Let us denote the space of \textit{gradient configurations} by $\O^\nabla = \Z^{\vec L}$. Equip the integers $\mathbb{Z}$ with the power set as measurable structure.
Having done this, the measurable structure on the space $\Omega^{\nabla}$ is given by the product $\sigma$-algebra $\mathcal{F}^\nabla:=\sigma(\{ \eta_b \, \vert \, b \in \vec{L} \})$. Clearly $\nabla: (\Omega, \mathcal{F}) \rightarrow (\Omega^\nabla, \mathcal{F}^\nabla  )$ then becomes a measurable map.

For any fixed site $x \in V$ and given spin value $\omega_x \in \mathbb{Z}$, each gradient configuration $\zeta \in \Omega^\nabla$ (uniquely) determines a height configuration by the measurable map

\begin{equation} \label{assertion}
\varphi_{x, \omega_x}: \begin{cases}
\Omega^\nabla \rightarrow \Omega \\	
(\varphi_{x, \omega_x}(\zeta))_y = \omega_x + \sum_{b \in \Gamma(x, y)} \zeta_b,
\end{cases}
\end{equation}
where $\Gamma(x,y)$ is the unique path from $x$ to $y$.  From this we get the following two statements:
\begin{enumerate}
	 \item The linear map $\nabla: \mathbb{Z}^V \rightarrow \mathbb{Z}^{\vec{L}}$ is surjective and \label{cycle}
	 \item \label{connected}
	 The kernel of $\nabla$ is given by the spatially homogeneous configurations.
\end{enumerate}
Therefore we have the identification
\begin{equation} \label{GrSp}
\O^\nabla = \Z^{\vec L}=\mathbb{Z}^V/\mathbb{Z}.
\end{equation}
Here, $=$ is meant in the sense of isomorphy between Abelian groups. Endowing $\mathbb{Z}^V/\mathbb{Z}$ with the final $\sigma$-algebra generated by the respective coset projection we can also regard this isomorphy as an isomorphy between measurable spaces due to measurability of the maps $\varphi_{x, \omega_x}$ and $\nabla$.

Note that statement \eqref{cycle} above relies on the absence of loops in trees. For gradient configurations on lattices in more than one dimension a further \textit{plaquette condition} is needed (see \cite{FS97}). In contrast to this, its following statement \eqref{connected} is based on connectedness of the tree. Therefore for any finite subtree $\Lambda \in \mathcal{S}$ the isomorphy \eqref{GrSp} between measurable spaces restricts to an isomorphy between $\mathbb{Z}^\Lambda / \mathbb{Z}$ and $\mathbb{Z}^{\{b \in \vec L \mid b \subset \Lambda\}}$ and $\mathbb{Z}^\Lambda / \mathbb{Z}$, where the sets are endowed with the respective final and product $\sigma$-algebra.

Further note that for any $w \in V$ the bijection
\begin{equation} \label{PushF}
\begin{cases}
\mathbb{Z}^V \rightarrow \mathbb{Z}^{\vec{L}} \times \mathbb{Z} \\
\omega=(\omega_x)_{x \in V} \mapsto (\nabla \omega, \omega_w )
\end{cases}
\end{equation}
is an isomorphism with respect to the product $\sigma$-algebra on $ \mathbb{Z}^{\vec{L}} \times \mathbb{Z}$, where the inverse map is given by \eqref{assertion}. In the following, this will allow us to easily identify any measure on $\mathbb{Z}^V$ with its push forward on the space $\mathbb{Z}^{\vec{L}}\times \mathbb{Z}$.

{\it Gibbs measure:} Recall that the set of height configurations $\Omega:=\mathbb{Z}^V$ was endowed with the product $\sigma$-algebra $\otimes_{i \in V}2^\mathbb{Z}$, where $2^\mathbb{Z}$ denotes the power set of $ \mathbb{Z}$.  Then for any $\Lambda \subset V$ consider the coordinate projection map $\sigma_\Lambda: \mathbb{Z}^V \rightarrow \mathbb{Z}^\Lambda$ and the $\sigma$-algebra $\mathcal{F}_\Lambda:=\sigma(\omega_\Lambda)$ of cylinder sets on $\mathbb{Z}^V$ generated by the map $\omega_\Lambda$.

Now we are ready to define Gibbs measures on the space of height-configurations for the model (\ref{nu1}) on a
Cayley tree. Let $\n=\{\n(i)>0, i\in \mathbb Z\}$ be a $\sigma$-finite  positive fixed a-priori
measure, which in the following we will always assume to be the counting measure.\\

  Gibbs measures are built within the DLR framework by describing conditional probabilities w.r.t. the outside of finite sets, where a boundary condition is frozen. One introduces a so-called Gibbsian specification $\gamma$ so that any Gibbs measure $\mu \in \mathcal{G}(\gamma)$ specified by $\gamma$ verifies
\begin{equation}\label{DLR_eq}
\mu (A | \mathcal{F}_{\Lambda^c}) = \gamma_\Lambda(A | \cdot) \quad \mu-{\rm a.s.}
\end{equation}
for all $\Lambda \in \mathcal{S}$ and $A \in \mathcal{F}$. The Gibbsian specification associated to a potential $\Phi$ is given at any inverse temperature $\beta >0$, for any boundary condition $\omega \in \Omega$ as
\begin{equation}\label{DLR-Gibbs}
 \gamma_\Lambda(A | \omega) = \frac{1}{Z_\Lambda^{\beta, \Phi }} \int e^{- \beta H^\Phi_\Lambda(\sigma_\Lambda \omega_{\Lambda^c})} \mathbf{1}_A(\sigma_\Lambda \omega_{\Lambda^c} ) \nu^{\otimes \Lambda}(d \sigma_\Lambda),
\end{equation}
where the partition function $Z_\Lambda^{\beta, \Phi }$ -- that has to be non-null and convergent in this countable infinite state-space context (this means that $\Phi$ is $\nu$-admissible in the terminology of \cite{Ge})-- is the standard normalization whose logarithm is often related to pressure or free energy. \\

In our SOS-model on the Cayley tree $\Phi$ is the unbounded nearest neighbour potential with \\$\Phi_{\{x,y\}}(\omega_x,\omega_y)=  \vert \omega_x-\omega_y \vert$ and $\Phi_{\{x\}} \equiv 0$, so $\gamma$ is a \textit{Markov specification} in the sense that
\begin{equation} \label{MSpec}
\gamma_\Lambda(\omega_\Lambda = \zeta | \cdot) \text{ is } \mathcal{F}_{\partial \Lambda}\text{-measurable for all } \Lambda \subset V \text{ and } \zeta \in \mathbb{Z}^\Lambda.
\end{equation}

In order to build up \textit{gradient specifications} from the Gibbsian specifications defined above, we need to consider the following: Due to the absence of loops in trees, for any finite $\Lambda \subset \mathbb{Z}$ the complement $\Lambda^c$ is not connected but consists of at least two connected components where each of these contains at least one element of $\partial \Lambda$. This means that the gradient field outside $\Lambda$ does not contain any information on the relative height of the boundary $\partial \Lambda$ (which is to be understood as an element of $\mathbb{Z}^{\partial \Lambda} / \mathbb{Z}$). More precisely, let $cc(\Lambda^c)$ denote the number of connected components in $ \Lambda^c$ and note that $2 \leq cc(\Lambda^c) \leq \vert \partial \Lambda \vert $.

Applying \eqref{assertion} to each connected component, an 
analogue to 
 \eqref{GrSp} becomes 
\begin{equation} \label{relboundary}
\mathbb{Z}^{\{b \in \vec{L} \, \vert \, b \subset \Lambda^c \}} \times (\mathbb{Z}^{\partial \Lambda}/\mathbb{Z}) \supset \mathbb{Z}^{\{b \in \vec{L} \, \vert \, b \subset \Lambda^c \}} \times (\mathbb{Z}^{cc( \Lambda^c)}/\mathbb{Z}) =\mathbb{Z}^{\Lambda^c} / \mathbb{Z} \subset \mathbb{Z}^V / \mathbb{Z}.
\end{equation}
where "$=$" is in the sense of isomorphy between measurable spaces.
For any $\eta \in \Omega^\nabla=\mathbb{Z}^V/\mathbb{Z}$ let $\lfloor \eta \rfloor_{\partial \Lambda} \in \mathbb{Z}^{\partial \Lambda}/\mathbb{Z}$ denote the image of $\eta$ under the coordinate projection $\mathbb{Z}^V/\mathbb{Z} \rightarrow \mathbb{Z}^{\partial \Lambda}/\mathbb{Z}$ with the latter set endowed with the final $\sigma$-algebra generated by the coset projection.
Set \begin{equation}
\mathcal{F}^\nabla_{\L}:= \sigma( (\eta_b)_{b \subset {\L}^c}) \subset \mathcal{T}^\nabla_{\L}:=\s( (\eta_b)_{b \subset {\L}^c }, \left[\eta\right]_{\partial \L}).
\end{equation}
Then $\mathcal{T}^\nabla_{\L}$ contains all information on the gradient spin variables outside $\Lambda$ and also information on the relative height of the boundary $\partial \Lambda$.
By \eqref{relboundary} we have that for any event $A \in \mathcal{F}^\nabla$ the $\mathcal{F}_{\Lambda^c}$-measurable function $ \gamma_\Lambda(A | \cdot)$ is also measurable with respect to $\mathcal{T}^\nabla_{\L}$, but in general not with respect to $\mathcal{F}^\nabla_{\L}$. These observations lead to the following:
\begin{defn}\label{d1}
	The {\em gradient Gibbs specification} is defined as the family of probability kernels $(\gamma'_\L)_{\L \subset \subset V}$ from $(\Omega^\nabla, \mathcal{T}_\L^\nabla)$ to $(\Omega^\nabla, \mathcal{F}^\nabla)$ such that
	\begin{equation}\begin{split}\label{grad}
	\int F(\rho) \gamma'_\L(d\rho \mid \zeta) = \int F(\nabla \varphi) \gamma_\L(d\varphi\mid\omega)
	\end{split}
	\end{equation}
	for all bounded $\mathcal{F}^{\nabla}$-measurable functions $F$, where $\omega \in \Omega$ is any height-configuration with $\nabla \omega = \zeta$.
\end{defn}
Using the sigma-algebra
$\mathcal{T}_{\L}^\nabla$, this is now a proper and consistent family of probability kernels, i.e.
\begin{equation}
\gamma'_\L(A \mid \zeta) = \mathbf{1}_A(\zeta)
\end{equation}
for every $A \in \mathcal{T}_\L^\nabla$ and $\gamma'_\Delta \gamma'_\L = \gamma'_\Delta$ for any finite volumes $\L, \Delta \subset V$ with $\L \subset \Delta$. The proof is similar to the situation of regular (local) Gibbs specifications \cite[Proposition 2.5]{Ge}.

Let $\mathcal{C}_b(\Omega^\nabla)$ be the set of bounded functions on $\Omega^\nabla$.
Gradient Gibbs measures will now be defined in the usual way by having their conditional probabilities outside finite regions prescribed by the gradient Gibbs specification:

\begin{defn} \label{GGM} A measure $\nu \in \mathcal{M}_1(\O^\nabla)$ is called a {\em gradient Gibbs measure (GGM)} if it satisfies the DLR equation
	\begin{equation}\label{DLR}
	\int \nu (d\zeta)F(\zeta)=\int \nu (d\zeta) \int \gamma'_{\L}(d\tilde\zeta \mid \zeta) F(\tilde\zeta)
	\end{equation}
for every finite $\Lambda \subset V$ and for all $F \in \mathcal{C}_b(\Omega^\nabla)$. The set of gradient Gibbs measures will be denoted by $\mathcal{G}^\nabla(\gamma)$. 
\end{defn}
\textit{Construction of GGMs via boundary laws:}

In what follows we may assume the a-priori measure $\nu$ on $\mathbb{Z}$ to be the counting measure.
On trees with nearest-neighbours potentials $\Phi$ such as the one we consider here, it is possible to use the natural orientations of edges to introduce tree-indexed Markov chains. These are probability measures $\mu$ having the property that for $\mu$-a.e. oriented edges $\langle xy \rangle$ and any $\omega_y \in E$,
$$
\mu(\sigma_y = \omega_y | \mathcal{F}_{(-\infty, xy)}) = \mu(\sigma_y = \omega_y | \mathcal{F}_x),
$$
where $$(-\infty, xy) := \{ w \in V \mid \langle x,y \rangle \in \vec L_w \},$$ denotes the past of the edge $\langle x,y \rangle$.
One can associate to $\mu$ a  transition matrix defined to be any stochastic matrix $P=(P_{xy})_{\langle xy \rangle}$ satisfying for all $\omega_y \in E$
$$
\mu(\sigma_y= \omega_y | \mathcal{F}_x ) = P_{xy}( \sigma_x,\omega_y) \; \; \mu-{\rm a.s.}
$$

For $n.n.$ interaction potential $\Phi=(\Phi_b)_b$, where bonds are denoted $b=\langle x,y \rangle$, one first defines symmetric transfer matrices $Q_b$ following the terminology of Cox \cite{Cox} or Zachary \cite{Z1,Z2} (see also \cite{Ge}). Setting
$$
Q_b(\omega_b) = e^{- \big(\Phi_b(\omega_b) + | \partial x|^{-1} \Phi_{\{x\}}(\omega_x) + |\partial y |^{-1} \Phi_{\{y\}} (\omega_y) \big)}
$$
one can rewrite the Gibbsian specification as
$$
\gamma_\Lambda^\Phi(\sigma_\Lambda = \omega_\Lambda | \omega) = (Z_\Lambda^\Phi)(\omega)^{-1} \prod_{b \cap \Lambda \neq \emptyset} Q_b(\omega_b).
$$
If for any bond $b=\langle x,y \rangle$ the transfer operator $Q_b(\omega_b)$ is a function of gradient spin variable $\zeta_b=\omega_y-\omega_x$ we call the underlying potential $\Phi$ a \textit{gradient interaction potential}. \\
Now we note the following:
On the one hand, each extreme Gibbs measure on a tree with respect to a Markov specification is a tree-indexed Markov chain (Theorem 12.6 in \cite{Ge}). On the other hand (Lemma 3.1 in\cite{Z1}), a measure $\mu$ is a Gibbs measure with respect to a nearest neighbour potential $\Phi$ with associated family of transfer matrices $(Q_b)_{b \in L}$ iff its marginals at any finite volume $ \Lambda \subset V$ are of the form
\begin{equation} \label{FormOfGM}
\mu(\sigma_{\Lambda \cup \partial \Lambda}=\omega_{\Lambda \cup \partial \Lambda})=c_\Lambda(\omega_{\partial \Lambda}) \prod_{b \cap \Lambda \neq \emptyset} Q_b(\omega_b)
\end{equation}
for some function $c_\Lambda: \partial \Lambda \rightarrow \mathbb{R}_+$.
Taking this into account leads to the concept of boundary laws that allows to describe the Gibbs measures that are Markov chains on trees.

\begin{defn}\label{def:bl}
	A family of vectors $\{ l_{xy} \}_{\langle x,y \rangle \in \vec E}$ with $l_{xy} \in (0, \infty)^\Z$ is called a {\em boundary law for the transfer operators $\{ Q_b\}_{b \in L}$} if for each $\langle x,y \rangle \in \vec L$ there exists a constant  $c_{xy}>0$ such that the consistency equation
	\begin{equation}\label{eq:bl}
	l_{xy}(\omega_x) = c_{xy} \prod_{z \in \partial x \setminus \{y \}} \sum_{\omega_z \in \Z} Q_{zx}(\omega_x,\omega_z) l_{zx}(\omega_z)
	\end{equation}
	holds for every $\omega_x \in \Z$. A boundary law is called to be {\em $q$-periodic} if $l_{xy} (\omega_x + q) = l_{xy}(\omega_x)$ for every oriented edge $\langle x,y \rangle \in \vec L$ and each $\omega_x \in \Z$.
	
\end{defn}

In our unbounded discrete context, there is as in the finite-state space context, a one-to-one correspondence between boundary laws and tree-indexed Markov chains, but for some boundary laws only, the ones that are {\em normalisable} in the sense of Zachary \cite{Z1,Z2}.

\begin{defn}[Normalisable boundary laws] \label{nBL}
A boundary law $l$ is said to be {\em normalisable} if and only if
\begin{equation}\label{Norm}
\sum_{\omega_x \in \Z} \Big( \prod_{z \in \partial x} \sum_{\omega_z \in \Z} Q_{zx}(\omega_x,\omega_z) l_{zx}(\omega_z) \Big) < \infty
\end{equation} at any $x \in V$.
\end{defn}

The correspondence now reads the following:
\begin{thm}[Theorem 3.2 in \cite{Z1}] \label{BLMC}
For any Markov specification $\gamma$ with associated family of transfer matrices $(Q_b)_{b \in L}$  we have
\begin{enumerate}
\item Each {\bf normalisable} boundary law $(l_{xy})_{x,y}$ for $(Q_b)_{b \in L}$ defines a unique tree-indexed Markov chain $\mu \in \mathcal{G}(\gamma)$ via the equation given for any connected set $\Lambda \in \mathcal{S}$
\begin{equation}\label{BoundMC}
\mu(\sigma_{\Lambda \cup \partial \Lambda}=\omega_{\Lambda \cup \partial \Lambda}) = (Z_\Lambda)^{-1} \prod_{y \in \partial \Lambda} l_{y y_\Lambda}(\omega_y) \prod_{b \cap \Lambda \neq \emptyset} Q_b(\omega_b),
\end{equation}
where for any $y \in \partial \Lambda$, $y_\Lambda$ denotes the unique $n.n.$ of $y$ in $\Lambda$.
\item Conversely, every tree-indexed Markov chains $\mu \in \mathcal{G}(\gamma)$ admits a representation of the form (\ref{BoundMC}) in terms of a {\bf normalisable} boundary law (unique up to a constant positive factor).
\end{enumerate}
\end{thm}
\begin{rk}
The Markov chain $\mu$ defined in \eqref{BoundMC} has the transition probabilities
\begin{equation}\label{last}
P_{ij}(\omega_i,x)=\mu(\s_j = x \mid \s_i = \omega_{i})
= \frac{l_{ji}(x) Q_{ji}(x, \omega_i)}{\sum_y l_{ji}(y) Q_{ji}(y, \omega_i)}.
\end{equation}
The expressions \eqref{last}  may exist even in situations where the underlying boundary law $(l_{xy})_{x,y}$ is not normalisable in the sense of Definition \ref{nBL}. However, the Markov chain given by the so defined transition probabilities is in general not positively recurrent which means that it does not possess an invariant probability measure. More precisely if the Markov chain defined by \eqref{last} is of the form \eqref{BoundMC} (and hence of the form \eqref{FormOfGM}) then its underlying boundary law must be necessarily normalisable as one can see by considering \eqref{BoundMC} for $\Lambda=\{x\}, \, x \in V$.
Thus, there is no obvious extension of Theorem $\ref{BLMC}$ to non-normalisable boundary laws.

\end{rk}	
Let us now assume that $Q_b = Q$ for all $b\in L$ (this holds obviously true for the SOS model). We call a vector $l \in (0,\infty)^\Z$ a (spatially homogeneous) boundary law if there exists a constant $c>0$ such that
the consistency equation
\begin{equation}\label{bl12}
l(i) = c \left(\sum_{j \in \Z} Q(i,j) l(j) \right)^k
\end{equation}
is satisfied for every $i \in \Z$. \\
Now assume that the elements of the family $(Q_b)_{b \in L}$ do not depend on the bonds i.e. $Q_b = Q$ for all $b \in L$, i.e. the underlying potential is tree-automorphism invariant.

In the case of spatially homogeneous boundary laws the expression \eqref{Norm} in the definition of normalisability reads
\begin{equation*}
\begin{split}
\sum_{i \in \Z} \bigl(\, \sum_{j \in \Z} Q(i,j) l(j) \, \bigr)^{k+1}&=\sum_{i \in \Z}c^{-\frac{k+1}{k}} {\bigl(\, c(\sum_{j \in \Z} Q(i,j) l(j) )^k \, \bigr)^{\frac{k+1}{k}}}  \cr
&= c^{-\frac{k+1}{k}}\sum_{i \in \Z}(l(i))^\frac{k+1}{k},
\end{split}
\end{equation*}
which means that any spatially homogeneous normalisable boundary law is an element of the space $l^{1+\frac{1}{k}}$. Thus periodic spatially homogeneous boundary laws are never normalisable in the sense of Definition \ref{nBL}.\\
However, it is possible to assign (tree-automorphism invariant) Gradient Gibbs measures to spatially homogeneous $q$-periodic boundary laws to tree-automorphism invariant gradient interaction potentials.
The main idea consists in considering for any boundary law $(l_{xy})$ to a gradient interaction potential and any finite connected subset $\Lambda \subset V$ the (in general only $\sigma$-finite) measure $\mu_\Lambda$ on $(\mathbb{Z}^{\Lambda \cup \partial \Lambda}, \otimes_{i \in \Lambda \cup \partial \Lambda} 2^\mathbb{Z} )$ given by the assignment \eqref{BoundMC}, i.e.
\begin{equation}
\mu(\sigma_{\Lambda \cup \partial \Lambda}=\omega_{\Lambda \cup \partial \Lambda}) =  \prod_{y \in \partial \Lambda} l_{y y_\Lambda}(\omega_y) \prod_{b \cap \Lambda \neq \emptyset} Q_b(\omega_b).
\end{equation}
Then fix  any pinning site $w \in \Lambda$ and identify $\mu_\Lambda$ with its pushforward measure on $\mathbb{Z}^{\vec{L}} \times \mathbb{Z}$ under \eqref{PushF}. This measure has the marginals
\begin{equation}
\begin{split}
\mu_{\Lambda \cup \partial \Lambda}(\sigma_w=i \, , \, \eta_{\Lambda \cup \partial \Lambda}=\zeta_{\Lambda \cup \partial \Lambda })& =\mu_{\Lambda \cup \partial \Lambda}(\sigma_w=i)\mu_{\Lambda \cup \partial \Lambda}( \eta_{\Lambda \cup \partial \Lambda}=\zeta_{\Lambda \cup \partial \Lambda } \mid \sigma_w=i)  \cr &=\mu_{\Lambda \cup \partial \Lambda}(\sigma_w=i) \prod_{y \in \partial \Lambda}{l_{yy_\Lambda}}(i+\sum_{b \in \Gamma(w,y)}\zeta_b)\prod_{b \cap \Lambda \neq \emptyset}Q_b(\zeta_b).
\end{split}
\end{equation}
If the boundary law $l$ is assumed to be $q$-periodic, then $\mu_{\Lambda \cup \partial \Lambda}( \eta_{\Lambda \cup \partial \Lambda}= \cdot \mid \sigma_w=i)$ will depend on $i$ only modulo $q$. For any class label $s \in \mathbb{Z}_q$ this allows us to obtain a probability measure $\nu_{w,s}$ on $\mathbb{Z}^{\{b \in \vec L \mid b \subset \Lambda\}}$ by setting
\begin{equation}
\begin{split}
\nu_{w,s}(\eta_{\Lambda \cup \partial \Lambda}=\zeta_{\Lambda \cup \partial \Lambda})&:=Z^\Lambda_{w,s}\mu_{\Lambda \cup \partial \Lambda}( \eta_{\Lambda \cup \partial \Lambda}= \zeta_{\Lambda \cup \partial \Lambda} \mid \sigma_w=s) \cr
&=Z^\Lambda_{w,s}\prod_{y \in \partial \Lambda} l_{yy_\L}\Bigl (T_q(
s+\sum_{b\in \Gamma(w,y)}\zeta_b)
\Bigr) \prod_{b \cap \Lambda \neq \emptyset}Q_b(\zeta_b),
\end{split}
\end{equation}
where $Z^\Lambda_{w,s}$ is a normalization constant and $T_q: \mathbb{Z} \rightarrow \mathbb{Z}_q$ denotes the coset projection.
Then one can show the following:
\begin{thm}[Theorem 3.1 in \cite{KS}] \label{Mar}
	Let $(l_{<xy>})_{<x,y> \in \vec L}$ be any $q$-periodic boundary law to some gradient interaction potential. Fix any site $w \in V$ and any class label $s \in \mathbb{Z}_q$. Then the definition
	\begin{equation}
	\nu_{w,s}(\eta_{\Lambda \cup \partial \Lambda}=\zeta_{\L\cup\partial\L})
	=Z^\Lambda_{w,s} \prod_{y \in \partial \Lambda} l_{yy_\L}\Bigl (T_q(
	s+\sum_{b\in \Gamma(w,y)}\zeta_b)
	\Bigr) \prod_{b \cap \Lambda \neq \emptyset}
	Q_b(\zeta_b),
	\end{equation}
	where $\Lambda$ with $w \in  \L \subset V$ is any finite connected set, $\zeta_{\L\cup\partial\L} \in \Z^{\{b \in \vec L \mid b \subset (\L \cup \partial\L)\}}$ and $\Z^\Lambda_{w,s}$ is a normalization constant, gives a consistent family of probability measures on the gradient space $\Omega^\nabla$.
	The measures $\nu_{w,s}$ will be called pinned gradient measures.
\end{thm}
By construction, the pinned gradient measures $\nu_{w,s}$ on $\Omega^\nabla$ have a \textit{restricted gradient} (Gibbs) property in the sense that the DLR-equation \eqref{DLR} holds for any finite $\Lambda \subset V$ which does not contain the pinning site $w$ (for details see \cite{KS}).
If the $q$-periodic boundary law is now additionally spatially homogeneous and the underlying potential is tree-automorphism invariant then it is possible to obtain a tree-automorphism invariant probability measure $\nu$ on the the gradient space by mixing the pinned gradient measures over an appropriate distribution on $\mathbb{Z}_q$. In this case, the restricted gradient Gibbs property of each of the pinned gradient measures leads to the Gibbs property of the measure $\nu$.

A useful representation of the finite-volume marginals of the resulting GGM is given in the following theorem:
\begin{thm}[Theorem 4.1, Remark 4.2 in \cite{KS}] \label{thm: constrGGM}	
Let $l$ be any spatially homogeneous $q$-periodic boundary law to a tree-automorphism invariant gradient interaction potential on the Cayley tree. Let $\Lambda \subset V$ be any finite connected set and let $w\in \Lambda$ be any vertex. Then the measure $\nu $ with marginals given by
\begin{equation}
\nu (\eta_{\L\cup\partial \L} = \zeta_{\L\cup\partial\L}) = Z_\L \ \left(\sum_{s\in\Z_q}  \prod_{y \in \partial\L} l \big(s + \sum_{b \in \Gamma(w,y)} \zeta_{b}\big)  \right)\prod_{b \cap \L \neq \emptyset} Q(\zeta_b),
\end{equation}
where $Z_\L$ is a normalisation constant, defines a spatially homogeneous GGM on $\Omega^\nabla$.
\end{thm}
\begin{rk} \label{ident}
Setting $n^w_i(\zeta_{\Lambda \cup \partial \Lambda }):=\vert \,  \{y \in \partial \Lambda \mid \sum_{b \in \Gamma(w,y)} \zeta_b \equiv i \mod q \} \, \vert$ the marginals of the measure $\nu$ defined in Theorem \ref{thm: constrGGM} can be written in the form:
\begin{equation}
\begin{split}
\nu (\eta_{\L\cup\partial \L} = \zeta_{\L\cup\partial\L}) &= Z_\L(\sum_{j \in \mathbb{Z}_q}\prod_{i \in \mathbb{Z}_q}{l_i^{n^w_{i+j}(\zeta_{\Lambda \cup \partial \Lambda })}} ) \prod_{b \cap \Lambda \neq \emptyset}Q(\zeta_b) \cr
&=Z_\L(\sum_{j \in \mathbb{Z}_q}\prod_{i \in \mathbb{Z}_q}{l_{i+j}^{n^w_i(\zeta_{\Lambda \cup \partial \Lambda })}} ) \prod_{b \cap \Lambda \neq \emptyset}Q(\zeta_b).
\end{split}
\end{equation}
This representation directly shows that two periodic boundary laws will lead to the same GGM if one is obtained from the other by a cyclic permutation or multiplication with a positive constant. \\
To obtain sufficient criteria for two GGM $\nu^l$ and $\nu^{\tilde{l}}$ associated to two distinct periodic boundary laws $l$ and $\tilde{l}$ with $l_0=\tilde{l}_0=1$ being distinct we first observe that $\nu^l = \nu^{\tilde{l}}$ if and only if
\begin{equation}
\begin{split}
\frac{\sum_{j \in \mathbb{Z}_q}\prod_{i \in \mathbb{Z}_q}{l_i^{n^w_{i+j}(\zeta_{\Lambda \cup \partial \Lambda })}}}{\sum_{j \in \mathbb{Z}_q}\prod_{i \in \mathbb{Z}_q}{\tilde{l}_i^{n^w_{i+j}(\zeta_{\Lambda \cup \partial \Lambda })}}}&=\frac{Z_{\Lambda}^{\tilde{l}}}{Z_{\Lambda}^l}
\end{split}
\end{equation}
for all finite subtrees $\Lambda$ and $\zeta_{\Lambda \cup \partial \Lambda } \in \mathbb{Z}^{\{b \in L \mid b \subset  \Lambda \cap \partial \Lambda\}}$. \\
Thus  $\nu^l = \nu^{\tilde{l}}$ if and only if for any finite subtree $\Lambda$ there is a constant $c(\Lambda)>0$ such that
\begin{equation} \label{necessary}
\sum_{j \in \mathbb{Z}_q}\prod_{i \in \mathbb{Z}_q}{l_i^{n_{i+j}}}=c(\Lambda) \,\sum_{j \in \mathbb{Z}_q}\prod_{i \in \mathbb{Z}_q}{\tilde{l}_i^{n_{i+j}}}
\end{equation}
for all vectors $(n_0,n_2, \ldots, n_{q-1}) \in \mathbb{N}_0$ with $\sum_{i \in \mathbb{Z}_q}n_i = \vert \partial \Lambda \vert $.
If we take a single-bond volume $\Lambda=\{b\}$, where $b \in L$,
we obtain the marginal
\begin{equation} \label{1-bound: necessary}
\begin{split}
\nu(\eta_b=\zeta_b)=Z_b\sum_{s \in \mathbb{Z}_q}l(s)l(s+\zeta_b)Q(\zeta_b)=Z_b(\sum_{j \in \mathbb{Z}_q}\prod_{i \in \mathbb{Z}_q}{l_i^{n_{i+j}(\zeta_b)}} )Q_b(\zeta_b).
\end{split}
\end{equation}
From this we get that if $\nu^l = \nu^{\tilde{l}}$ then condition (\ref{necessary}) is fulfilled for all vectors  \\$(n_0,n_1, \ldots, n_{q-1}) \in \{0,1,2\}^q$ with $\sum_{i \in \mathbb{Z}_q}n_i = 2$.
%
\end{rk}
We will now conclude some statements on identifiability of GGM with respect to the class of boundary laws which we will describe in the following section.
\begin{lemma} \label{lem: 2-per}
	Let $l$ and $\tilde{l}$ be two $2$-periodic boundary laws with $l_0=\tilde{l}_0=1$. Denote $l_1=a_1$, $\tilde{l}_1=a_2$. \\Then
	\begin{equation}
	\nu^l = \nu^{\tilde{l}} \quad \text{if and only if} \quad a_1=a_2 \text{ or } a_1a_2=1
	\end{equation}
\end{lemma}
\begin{proof}
	Let us first prove that $\nu^l = \nu^{\tilde{l}}$ if $ a_1=a_2 \text{ or } a_1a_2=1$.
	Using the marginals representation given in Remark \ref{ident}
	we have that $\nu^l=\nu^{\tilde{l}}$ if and only if
	\begin{equation} \label{2-per}
	\frac{\sum_{j \in \{0,1\}}\prod_{i \in \{0,1\}}{l_i^{n^w_{i+j}(\zeta_{\Lambda \cup \partial \Lambda })}}}{\sum_{j \in \{0,1\}}\prod_{i \in \{0,1\}}{\tilde{l}_i^{n^w_{i+j}(\zeta_{\Lambda \cup \partial \Lambda })}}} =const(\Lambda \cup \partial \Lambda)
	\end{equation}
	for any finite subtree $\Lambda$ and $\zeta_{\Lambda \cup \partial \Lambda } \in \mathbb{Z}^{\{b \in L \mid b \subset  \Lambda \cup \partial \Lambda \}}$.
	Let $n=\vert  \partial \Lambda  \vert \geq 3$ then the vectors $(n_0,n_1)(\zeta_{\L\cup\partial\L})$ are of the form $(n-m,m)$ for some integer $0 \leq m \leq n$.
	Inserting this into (\ref{2-per}) we conclude that $\nu^l=\nu^{\tilde{l}}$ if and only if for all $n \geq 3$ which can be realized as the number of points in the boundary of a finite subtree and any $0 \leq m_1,m_2 \leq n$ we have:
	\begin{equation}
	\frac{a_1^{n-m_1}+a_1^{m_1}}{a_2^{n-m_1}+a_2^{m_1}}=\frac{a_1^{n-m_2}+a_1^{m_1}}{a_2^{n-m_2}+a_2^{m_2}}.
	\end{equation}
	We may further assume $m_1=0$ and write $m=m_2$. Then this equation reduces to
	\begin{equation} \label{2-per finEq}
	\frac{a_1^{n}+1}{a_2^{n}+1}=\frac{a_1^{n-m}+a_1^{m}}{a_2^{n-m}+a_2^{m}},
	\end{equation}
	which holds true if $a_2=(a_1)^{-1}$. \\
To prove the other direction we must show that $a_1=a_2$ and $a_1a_2=1$ are the only solutions to the system (\ref{2-per finEq}):\\
For any $x>0$ set $f_{(n,m)}(x):=\frac{x^n+1}{x^{n-m}+x^m}$. Then \eqref{2-per finEq} is equivalent to $f_{(n,m)}(a_1)=f_{(n,m)}(a_2)$. Consider any $0<m<n$. Clearly $f_{(n,m)}$ is continuous, strictly decreasing on $(0,1)$ and strictly increasing on $(1, \infty)$, which means that for any $x \in (0,1 \rbrack$ there is at most one $y \in \lbrack 1, \infty)$ with $f_{(n,m)}(x)=f_{(n,m)}(y)$.
Since $f_{(n,m)}(x)=f_{(n,m)}(\frac{1}{x})$ we have that $f_{(n,m)}(a_1)=f_{(n,m)}(a_2)$ if and only if $a_1=a_2$ or $a_1a_2=0$.

\end{proof}
\begin{lemma} \label{lem: 4-per}
Consider any $4$-periodic boundary law of the type
\begin{equation*}
l_i^{(a,b)}=\left\{ \begin{array}{lll}
1, \ \ \mbox{if} \ \ i \equiv 0 \ \ \mbox{or} \ \ 2 \mod 4 \\
a, \ \ \mbox{if} \ \ i \equiv 1 \mod 4 \\
b, \ \ \mbox{if} \ \ i \equiv 3 \mod 4 \\
\end{array}
\right.
\end{equation*}
and denote the associated GGM by $\nu^{(a,b)}$. Let $(a_1,b_1) ,(a_2,b_2)$ be two such boundary laws. If $\nu^{(a_1,b_1)}=\nu^{(a_2,b_2)}$ then necessarily
\begin{equation*}
\begin{split}
&a_1+b_1 = a_2+b_2 \quad \text{or } \cr 
&(a_i+b_i)(a_j+b_j) = 4.
\end{split}
\end{equation*}
\end{lemma}
\begin{proof}
Consider the marginal on a set $\Lambda:=\{b\}$, where $b \in L$ is any edge. Inserting the vectors $(n_0,n_1,n_2,n_3)=(2,0,0,0), \, (1,1,0,0)$ and $(1,0,1,0)$ into (\ref{necessary}) we conclude that if $\nu^{(a_1,b_1)}=\nu^{(a_2,b_2)}$ then there is some constant $c>0$ with
\begin{enumerate}
	\item $a_1^2+b_1^2+2=c(a_2^2+b_2^2+2)$,
	\item $a_1+b_1=c(a_2+b_2) \quad \text{and}$ 
	\item$ 1+a_1b_1=c(1+a_2b_2)$.
\end{enumerate}
Adding twice the third equation to the first we obtain
\begin{equation*}
(a_1+b_1)^2+4=c((a_2+b_2)^2+4),
\end{equation*}
which in combination with $(2)$ gives
\begin{equation}
\frac{(a_1+b_1)^2+4}{(a_2+b_2)^2+4}=\frac{a_1+b_1}{a_2+b_2}.
\end{equation}
Setting $x:=a_1+b_1$ and $y:=a_2+b_2$ leads to the equation $\frac{x^2+4}{y^2+4}=\frac{x}{y}$ which is equivalent to
\begin{equation*}
(x-y)(xy-4)=0.
\end{equation*}
This completes the proof.
\end{proof}
\begin{lemma} \label{lem: 3-per}
	Consider the $k$-regular tree, $k \geq 2$,  and a $3$-periodic boundary law of the type
	\begin{equation*}
	l_i^{(c)}=\left\{ \begin{array}{lll}
	1, \ \ \mbox{if} \ \ i \equiv 0 \, \mod 3 \\
	c, \ \ \text{else.} \\
	\end{array}
	\right.
	\end{equation*}
	Denote the associated GGM by $\nu^{(c)}$.
	Then the following holds true:
	\begin{enumerate}[a)]
	 \item $\nu^{(c_1)}= \nu^{(c_2)}$ if and only if for any $n \in \mathbb{N}$ we have $f_{(m_0,m_1,m_2)}(c_1)=f_{(m_0,m_1,m_2)}(c_2)$ for all $(m_0,m_1,m_2) \in {\{0,1, \ldots, n(k-1)+2)\}}^3$ with $m_0+m_1+m_2 = 2(n(k-1)+2)$, where \label{3-per: general}
			\[ f_{(m_0,m_1,m_2)}(x):=\frac{x^{m_0}+x^{m_1}+x^{m_2}}{1+2x^{n(k-1)+2}}, \quad x>0.
			\]
		\item 	$\nu^{(c_1)}= \nu^{(c_2)}$ if and only if $c_1=c_2$.
		\item The GGMs associated to the nontrivial members of this family of solutions are all different from the GGMs associated to the solutions given by the family of boundary laws defined in Lemma \ref{lem: 4-per}.
	\end{enumerate}
\end{lemma}
\begin{proof}
	The structure of the proof is similar to the proof of Lemma \ref{lem: 4-per}:
\begin{enumerate}[a)]
	\item First note that for any subtree of the $k$-regular tree with $n$ vertices we have $n(k+1)-2(n-1)=n(k-1)+2$ points in the outer boundary which follows by induction on $n$ (see \cite{Ro08}). Thus $\nu^{(c_1)}= \nu^{(c_2)}$ if and only if for each $n \in \mathbb{N}$ the equation \eqref{necessary} holds true for all $(n_0,n_1,n_2) \in {\mathbb{N}_0}^3$ with $n_0+n_1+n_2=n(k-1)+2$. This is equivalent to the existence of some $\lambda>0$ depending only on $k$ and $n$ with  \begin{equation*}
	c_1^{(n_0+n_1)}+c_1^{(n_0+n_2)}+c_1^{(n_1+n_2)}= \lambda 	(c_2^{(n_0+n_1)}+c_2^{(n_0+n_2)}+c_2^{(n_1+n_2)}).
	\end{equation*} 
	for all such vectors $(n_0,n_1,n_2)$. \\
	Setting $\begin{pmatrix}
	m_0 \\ m_1 \\m_2
	\end{pmatrix} := \begin{pmatrix}
	1 &1&0 \\ 1&0&1\\ 0&1&1
	\end{pmatrix}
	\begin{pmatrix}
	n_0 \\n_1 \\n_2
	\end{pmatrix}$, i.e.
	$\begin{pmatrix}
n_0 \\n_1 \\n_2 
	\end{pmatrix} = \frac{1}{2}\begin{pmatrix}
1&1&-1 \\ 1&-1&1 \\ -1 &1&1

	\end{pmatrix}
	\begin{pmatrix}
	m_0 \\m_1 \\m_2
	\end{pmatrix}$ this is equivalent to 
	 \begin{equation*}
	c_1^{m_0}+c_1^{m_1}+c_1^{m_2}= \lambda 	(c_2^{m_0}+c_2^{m_1}+c_2^{m_2})
	\end{equation*}
	for all
$(m_0,m_1,m_2) \in {\{0,1, \ldots, n(k-1)+2\}}^3$ with $m_0+m_1+m_2 = 2(n(k-1)+2)$. 
	Hence we have $\nu^{(c_1)}= \nu^{(c_2)}$ if and only if 
	\begin{equation*}
	\frac{c_1^{m_0}+c_1^{m_1}+c_1^{m_2}}{c_2^{m_0}+c_2^{m_1}+c_2^{m_2}}=\frac{c_1^{\tilde{m}_0}+c_1^{\tilde{m}_1}+c_1^{\tilde{m}_2}}{c_2^{\tilde{m}_0}+c_2^{\tilde{m}_1}+c_2^{\tilde{m}_2}}
	\end{equation*}
	for all vectors $(m_0,m_1,m_2), (\tilde{m}_0,\tilde{m}_1,\tilde{m}_2)$.
	Fixing $(\tilde{m}_0,\tilde{m}_1,\tilde{m}_2)=(0,n(k-1)+2,n(k-1)+2)$ this is equivalent to 
	\begin{equation*}
	\frac{c_1^{m_0}+c_1^{m_1}+c_1^{m_2}}{c_2^{m_0}+c_2^{m_1}+c_2^{m_2}}=\frac{1+2c_1^{n(k-1)+2}}{1+2c_2^{n(k-1)+2}}
	\end{equation*}
	for all $(m_0,m_1,m_2) \in {\{0,1, \ldots, n(k-1)+2\}}^3$ with $m_0+m_1+m_2 = 2(n(k-1)+2)$
	which proves the first statement. \\
	\item Consider a single-bond marginal 
	$\Lambda=\{b\}$, $b \in L$ and insert the vectors $(n_0,n_1, n_{2})=(2,0,0)$ and $(1,1,0)$ in \eqref{1-bound: necessary}. If $\nu^{(c_1)}=\nu^{(c_2)}$ then there is a constant $\lambda>0$ with 
	\begin{enumerate} [i)]
		\item $1+2c_1^2=\lambda(1+2c_2^2)$  and 
		\item $c_1^2+2c_1=\lambda(c_2^2+2c_2)$.
	\end{enumerate}
From this we obtain 
the polynomial equation in $c_1$:
\begin{equation}
c_1^2(4c_2-1)-2c_1(1+2c_2^2)+c_2(c_2+2)=0.
\end{equation}
Dividing out the linear term $(c_1-c_2)$ we arrive at 
\begin{equation} \label{3-per: explicit}
c_1(4c_2-1)-c_2-2=0.
\end{equation}

In the second step we will show that the assumption $\nu^{(c_1)}=\nu^{(c_2)}$ and $c_1 \neq c_2$ leads to a contradiction.  This will be done by considering \ref{3-per: general} for $n \rightarrow \infty$. Take any real numbers $0<a_0<a_1<a_2 < k-1$ where $a_0+a_1+a_2=2(k-1)$.
Then there is a sequence $(m_0(n),m_1(n),m_2(n))_{n \in \mathbb{N}}$  such that for all $n \in \mathbb{N}$ we have $(m_0(n),m_1(n),m_2(n)) \in {\{0,1, \ldots, n(k-1)+2)\}}^3$ and $m_0(n)+m_1(n)+m_2(n) = 2(n(k-1)+2)$ with the property that $\frac{m_i(n)}{n} \stackrel{n \rightarrow \infty}{\rightarrow} a_i$, $i \in \{0,1,2\}$. 

If $\nu^{(c_1)}=\nu^{(c_2)}$ and $c_1 \neq c_2$ then by \eqref{3-per: explicit} we have $c_2=\frac{c_1+2}{4c_1-1}$, so we may assume $\frac{1}{4}<c_1<1<c_2$. From \ref{3-per: general} we obtain
\begin{equation*}
\lim_{n \rightarrow \infty}\frac{1}{n}\log f_{(m_0(n),m_1(n),m_2(n))}(c_1)=\lim_{n \rightarrow \infty}  \frac{1}{n}\log f_{(m_0(n),m_1(n),m_2(n))}(c_2) 
\end{equation*} 
Hence, taking into account the assumption $0<c_1<1<c_2$ this implies
\begin{equation} \label{3-per: limit}
a_0\log(c_1)=(a_2-(k-1))\log(c_2),
\end{equation}
where the limiting behaviour of the l.h.s can be seen by writing $c_1^{m_0(n)}+c_1^{m_1(n)}+c_1^{m_2(n)}=c_1^{m_0(n)}(1+c_1^{m_1(n)-m_0(n)}+c_1^{m_2(n)-m_0(n)})$ and then inserting $m_i(n)=a_in+\varepsilon_{i,n}$ where $\varepsilon_{i,n} \stackrel{n \rightarrow \infty}{\rightarrow} 0$. The r.h.s. follows similarly.

Now \eqref{3-per: limit} is equivalent to
\begin{equation} \label{3-per: ldp}
\frac{\log (c_2)}{\log (c_1)}=\frac{a_0}{a_2-(k-1)}.
\end{equation} 
As $c_2$ is uniquely given by \eqref{3-per: explicit} and \eqref{3-per: ldp} holds true for all choices of $(a_0,a_1,a_2)$ in the allowed range, 
the assumption $\nu^{(c_1)}=\nu^{(c_2)}$ and $c_1 \neq c_2$ leads to a contradiction.
\item Let  $l^{(a,b)}$ denote any $4$-periodic boundary law as defined in Lemma \ref{lem: 4-per} and let $l^{(c)}$ be any $3$-periodic boundary law as defined above. We will consider each of them as a $12$-periodic boundary law. Take $\Lambda=\{b\}$, $b \in L$ and insert the vectors $(n_0,n_1, \ldots, n_{11})=(1,1,0,0,\ldots, 0)$ and $(1,0,0,1,0,0, \ldots,0)$ into \eqref{1-bound: necessary}. If $\nu^{(a,b)}=\nu^{(c)}$ then there is a constant $\lambda>0$ with
\begin{enumerate}[i)]
	\item $6(a+b)=4\lambda(2c+c^2)$  and 
	\item $6(a+b)=4 \lambda(1+2c^2)$
\end{enumerate}
From this we get $c^2-2c+1=0$ which leads to $c=1$.
\end{enumerate}
\end{proof}
\begin{rk}
Lemma \ref{lem: 2-per} can also be concluded from Lemma \ref{lem: 4-per} and the fact that two periodic boundary laws lead to the same GGM
if one is obtained from the other by cyclic permutations or multiplication with a positive constant.
\end{rk}
\section{Translation-invariant solutions}
In this section we calculate periodic solutions to the boundary law equation for the SOS-model.
First let $\beta>0$ be any inverse temperature and set $ \theta:= \exp(-\beta)<1$. The transfer operator $Q$ then reads $Q(i-j)=\theta^{\vert i-j \vert }$ for any $i,j \in \mathbb{Z}$ and a spatially homogeneous boundary law, now denoted by $z$, is any positive function on $\mathbb{Z}$ solving the system \eqref{bl12}, whose values we will denote by $z_i$ instead of $z(i)$. Further notice that a boundary law is only unique up to multiplication with any positive prefactor. Hence we may choose this constant in a way such that we have $z_0=1$. At last set $\mathbb{Z}_0:= \mathbb{Z} \setminus \{0\}$.
Taking into account these prerequisites the boundary law equation \eqref{bl12} now reads
\begin{equation}\label{nu11}
z_i=\left({\theta^{|i|}+
\sum_{j\in \mathbb Z_0}\theta^{|i-j|}z_j
\over
1+\sum_{j\in \mathbb Z_0}\theta^{|j|}z_j}\right)^k, \ \ i\in\mathbb Z_0.
 \end{equation}
\subsection{A simplification of the system (\ref{nu11})}
Let $\mathbf z(\theta)=(z_i=z_i(\theta), i\in \mathbb Z_0)$ be a solution to (\ref{nu11}).   Denote
\begin{equation}\label{lr}
l_i\equiv l_i(\theta)=\sum_{j=-\infty}^{-1}\theta^{|i-j|}z_j, \ \
r_i\equiv r_i(\theta)=\sum_{j=1}^{\infty}\theta^{|i-j|}z_j, \ \ i\in\mathbb Z_0.
\end{equation}
It is clear that each $l_i$ and $r_i$ can be a finite positive number or $+\infty$.
\begin{lemma}\label{l1} For each $i\in \mathbb Z_0$ we have
\begin{itemize}
\item $l_i<+\infty$ if and only if $l_0<+\infty$;

\item $r_i<+\infty$ if and only if $r_0<+\infty$.
\end{itemize}
\end{lemma}
\begin{proof}
The proof follows from the following equalities
\begin{equation}\label{li}
l_i=\left\{\begin{array}{ll}
\theta^i l_0+\sum_{j=i}^{-1}(\theta^{j-i}-\theta^{i-j})z_j, \ \ \mbox{if} \ \ i\leq -1\\[3mm]
\theta^i l_0,  \ \ \ \ \ \ \mbox{if} \ \ i\geq 1.
\end{array}
 \right.
\end{equation}
\begin{equation}\label{ri}
r_i=\left\{\begin{array}{ll}
\theta^{-i} r_0+\sum_{j=1}^{i}(\theta^{i-j}-\theta^{j-i})z_j, \ \ \mbox{if} \ \ i\geq 1\\[3mm]
\theta^{-i} r_0,  \ \ \ \ \ \ \mbox{if} \ \ i\leq -1.
\end{array}
 \right.
\end{equation}
\end{proof}


In what follows, we will always assume that $l_0<+\infty$ and $r_0<+\infty$.


Denoting $u_i=u_0\sqrt[k]{z_i}$ (for some $u_0>0$) from (\ref{nu11}) we get
$$
u_i=C\cdot (\dots+\theta^2u^k_{i-2}+\theta u^k_{i-1}+u_i^k+\theta u_{i+1}^k+\theta^2u^k_{i+2}+\dots), \ \ i\in \mathbb Z,$$
for some $C>0$.

This system can be written as

\begin{equation}\label{U}
u_i=C\left( \sum_{j=1}^{+\infty}\theta^ju_{i-j}^k+u_i^k+ \sum_{j=1}^{+\infty}\theta^ju_{i+j}^k\right), \ \ i\in \mathbb Z.
\end{equation}
\begin{pro} \label{pps: 1}
A vector $\mathbf u=(u_i,i\in \mathbb Z)$, with $u_0=1$,  is a solution to (\ref{U}) if and only if for $u_i \ \ (=\sqrt[k]{z_i})$ the following holds
\begin{equation}\label{V}
u_i^k={u_{i-1}+u_{i+1}-\tau u_i\over u_{-1}+u_{1}-\tau}, \ \ i\in \mathbb Z,
\end{equation}
where $\tau=\theta^{-1}+\theta=2\cosh(\beta)$.
\end{pro}
\begin{proof} {\sl Necessity.} From (\ref{U}) we get
$$u_{i-1}+u_{i+1}=$$ $$C\left(\sum_{j=1}^\infty \theta^ju_{i-1-j}^k+
\sum_{j=1}^\infty \theta^{j}u_{i+1-j}^k+u^k_{i-1}+u^k_{i+1}+
\sum_{j=1}^\infty \theta^ju_{i-1+j}^k+
\sum_{j=1}^\infty \theta^{j}u_{i+1+j}^k\right)=$$
 $$C\left(\theta^{-1}\sum_{{m=1\atop m=j+1}}^\infty \theta^mu_{i-m}^k-u^k_{i-1}+
\theta\sum_{{j'=1\atop j'=j-1}}^\infty \theta^{j'}u_{i-j'}^k+\theta u^k_i+u^k_{i-1}+\right.$$ $$\left.u^k_{i+1}+\theta
\sum_{{n=1\atop n=j-1}}^\infty \theta^nu_{i+n}^k+\theta u^k_i+\theta^{-1}
\sum_{{\bar j=1\atop \bar j=j+1}}^\infty \theta^{\bar j}u_{i+\bar j}^k-u^k_{i+1}\right)=$$
$$C\left((\theta^{-1}+\theta)\sum_{j=1}^\infty \theta^ju_{i-j}^k+2
\theta u_{i}^k+(\theta^{-1}+\theta)\sum_{j=1}^\infty \theta^ju_{i+j}^k\right)=$$$$
(\theta^{-1}+\theta)u_i+C(\theta-\theta^{-1})u^k_i.$$
Thus
\begin{equation}
\label{u0}
u_{i-1}+u_{i+1}-(\theta^{-1}+\theta)u_i=C(\theta-\theta^{-1})u_i^k, \ \ i\in \mathbb Z.
\end{equation}
Since $u_0=1$ dividing both sides of (\ref{u0}) to the equality of the case $i=0$ we get
 (\ref{V}).

{\sl Sufficiency.} Assume (\ref{V}) holds. Then we get (\ref{u0}) with some $C=\tilde C$. Write this equality for $i$ replaced by $i+1-j$, i.e.
 \begin{equation}
 \label{u1}
 u_{i-j}+u_{i-j+2}-(\theta^{-1}+\theta)u_{i+1-j}=\tilde C(\theta-\theta^{-1})u_{i+1-j}^k, \ \ i,j\in \mathbb Z.
 \end{equation}
 Multiply both sides of (\ref{u1}) by $\theta^j$ and sum over $j=1,2,\dots$. Here, absolute convergence of all occurring infinite sums is guaranteed by the assumption $l_0<+\infty$ and $r_0<+\infty$ and the fact that $\theta<1$. Then after rearrangement/simplifications we get
 $$
  \label{u2}
  \theta u_{i+1}-u_i=\tilde C(\theta-\theta^{-1})\sum_{j=1}^\infty\theta^j u_{i+1-j}^k.
  $$
  Dividing both sides of this equality by $\theta$ we get
 \begin{equation}
 \label{u2}
 u_{i+1}-\theta^{-1}u_i=\tilde C(\theta-\theta^{-1})\sum_{j=1}^\infty\theta^{j-1} u_{i+1-j}^k=
\tilde C(\theta-\theta^{-1})\left(\sum_{{m=1\atop m=j-1}}^\infty\theta^{m} u_{i-m}^k+u_i^k\right).
 \end{equation}

 Now rewrite (\ref{u0}) for $i$ replaced by $i+j$, $C$ is replaced by $\tilde C$  and multiply both sides of the obtained equality by $\theta^j$ then sum over $j=1,2,\dots$. After simplifications we get
   \begin{equation}
   \label{u3}
   \theta u_{i}-u_{i+1}=\tilde C(\theta-\theta^{-1})\sum_{j=1}^\infty \theta^ju_{i+j}^k.
   \end{equation}
 Adding (\ref{u2}) and (\ref{u3}) we get the $i$th equation of (\ref{U}) with $C$ replaced by $\tilde C$. For $u_0=1$ we get $\tilde C=C$.

\end{proof}

\begin{lemma}\label{l2}
If $l_0<+\infty$ and $r_0<+\infty$ then we have
$$l_0={\theta-u_{-1}\over u_{-1}+u_{1}-\tau},\ \ \
r_0={\theta-u_{1}\over u_{-1}+u_{1}-\tau}.$$
\end{lemma}
\begin{proof} Using (\ref{V}) we get
$$l_0=\sum_{j=-\infty}^{-1}\theta^{-j}z_j=
\sum_{j=-\infty}^{-1}\theta^{-j}u_j^k=
\sum_{j=-\infty}^{-1}\theta^{-j}{u_{j-1}+u_{j+1}-\tau u_j\over u_{-1}+u_{1}-\tau}.$$
Compute the following
$$\sum_{j=-\infty}^{-1}\theta^{-j}(u_{j-1}+u_{j+1}-\tau u_j)=$$
$$ \theta^{-1}\sum_{j=-\infty}^{-1}\theta^{-j+1}u_{j-1}+\theta \sum_{j=-\infty}^{-1}\theta^{-j-1}u_{j+1}-\tau\sum_{j=-\infty}^{-1}\theta^{-j}u_j=$$
$$ \theta-u_{-1}+(\theta^{-1}+\theta-\tau)\sum_{j=-\infty}^{-1}\theta^{-j}u_j.$$
Since $\theta^{-1}+\theta-\tau=0$ we get the formula of $l_0$. The case $r_0$ is similar.
\end{proof}

By this Lemma we have
\begin{equation}\label{1lr}
1+l_0+r_0={\theta-\theta^{-1}\over u_{-1}+u_1-\tau}.
\end{equation}

The equation (\ref{V}) can be separated into the following independent recurrent equations
\begin{equation}\label{L}
u_{-i-1}=(u_{-1}+u_1-\tau)u_{-i}^k+\tau u_{-i}-u_{-i+1}, \end{equation}
\begin{equation}\label{Re}
u_{i+1}=(u_{-1}+u_1-\tau)u_{i}^k+\tau u_{i}-
u_{i-1}, \end{equation}
where $i\geq 0$, $u_0=1$ and $u_{-1}$, $u_{1}$ are some initial numbers.
Note that for $i=0$ the above equations are trivially fulfilled for all values of $u_1$ and $u_{-1}$. Hence it suffices to consider \eqref{L} and \eqref{Re} for $i \geq 1$.


\subsection{A class of 4-Periodic solutions to (\ref{V})}


In this subsection we shall describe the two-parameter family of solutions to (\ref{V}) which have the form
\begin{equation}\label{up}
u_n=\left\{ \begin{array}{lll}
1, \ \ \mbox{if} \ \ n=0 \ \ \mbox{or} \ \ 2 \mod 4,\\[2mm]
a, \ \ \mbox{if} \ \ n=1 \mod 4\\[2mm]
b, \ \ \mbox{if} \ \ n=3 \mod 4,
\end{array}
\right.
\end{equation}
where $a$ and $b$ some positive numbers.
Such a solution defines a periodic two-side infinite sequence, i.e.
\begin{equation}
\label{pab}
..., a, 1, b, 1, a, 1, b, 1, a, 1, b,...
\end{equation}

The equations \eqref{L} and \eqref{Re} give the following system of equations
\begin{equation}
\label{ab}
\begin{array}{ll}
(a+b-\tau)b^k+\tau b-2=0\\[2mm]
(a+b-\tau)a^k+\tau a-2=0.
\end{array}
\end{equation}
For simplicity we consider the case $k=2$ and give full analysis of the system (\ref{ab}).\\

In case $k=2$ subtracting from the first equation of the system the second one we get
$$(b-a)[(a+b)^2-\tau(a+b)+\tau]=0.$$
Which gives three possibilities:
\begin{equation}
\label{b}
a=b,  \ \ \mbox{and} \ \  a={1\over 2}(\tau\pm \sqrt{\tau^2-4\tau})-b \ \ \mbox{for} \ \ \tau\geq 4.
\end{equation}

{\bf Case $a=b$}. In this case from the first equation of (\ref{ab}) we get
\begin{equation}
\label{a3}
2a^3-\tau a^2+\tau a-2=0.
\end{equation}
One easily gets the following solutions to this equation (recall that $\tau>2$):
\begin{itemize}
\item If $\tau \leq 6$ then the equation (\ref{a3}) has unique solution $a_0=1$.
\item If  $\tau>6$ then there are three solutions (see Fig. \ref{fn1})
$$a_0=1, \ \ a_1={1\over 4}(\tau-2-\sqrt{(\tau-2)^2-16}), \ \  a_2={1\over 4}(\tau-2+\sqrt{(\tau-2)^2-16}).$$
\end{itemize}
Note that these $2$-periodic solutions can be already found in \cite{KS} (recall that $\tau=2\cosh(\beta)$).\\
\begin{figure}
\includegraphics[width=9cm]{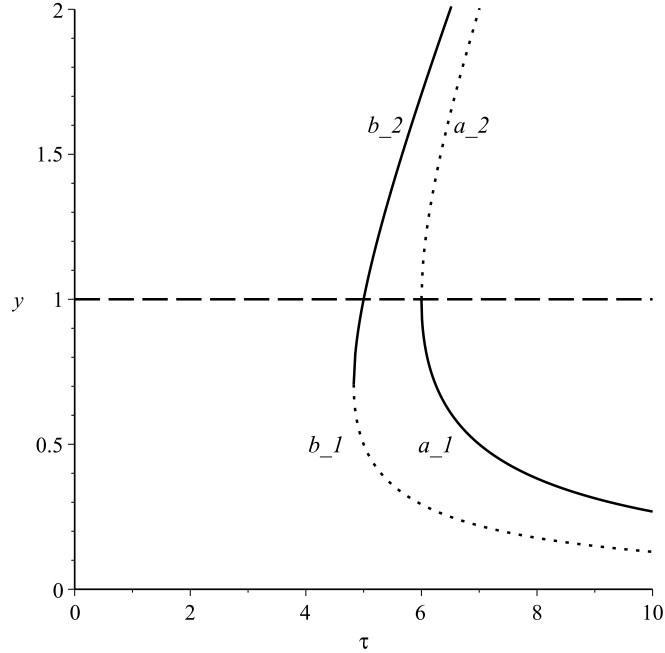}
\caption{ The graphs of the functions $a_1$, $a_2$ giving 2-periodic boundary laws. The graphs of the functions $b_1$, $b_2$ giving 3-periodic boundary laws. As $a_2=a_1^{-1}$, both 2-periodic boundary laws lead to the same GGM.\label{fn1}}
\end{figure}
{\bf Case  $a+b={1\over 2}(\tau+ \sqrt{\tau^2-4\tau})$}.
In this case from the second equation of (\ref{ab}) we get
$$(\tau-\sqrt{\tau^2-4\tau})a^2-2\tau a +4=0.$$
Which for $\tau\geq 4$ has the solutions
$$a_3={\tau-\sqrt{\tau^2-4\tau+4\sqrt{\tau^2-4\tau}}\over \tau-\sqrt{\tau^2-4\tau}}, \ \ a_4={\tau+\sqrt{\tau^2-4\tau+4\sqrt{\tau^2-4\tau}}\over \tau-\sqrt{\tau^2-4\tau}}. $$
Using (\ref{b}) we get $b_3=a_4$ and $b_4=a_3$.

{\bf Case  $a+b={1\over 2}(\tau- \sqrt{\tau^2-4\tau})$}.
In this case similarly as in previous case we obtain
\begin{equation*}
(\tau+\sqrt{\tau^2-4\tau})a^2-2\tau a+4=0
\end{equation*}
which
for $\tau\geq 2+2\sqrt{5}$ has the following solutions
$$a_5={\tau-\sqrt{\tau^2-4\tau-4\sqrt{\tau^2-4\tau}}\over \tau+\sqrt{\tau^2-4\tau}}, \ \ a_6={\tau+\sqrt{\tau^2-4\tau-4\sqrt{\tau^2-4\tau}}\over \tau+\sqrt{\tau^2-4\tau}}. $$
Using (\ref{b}) we get $b_5=a_6$ and $b_6=a_5$. Clearly all of these solutions are positive (see Fig. \ref{fn2}).
\begin{figure}
\includegraphics[width=8cm]{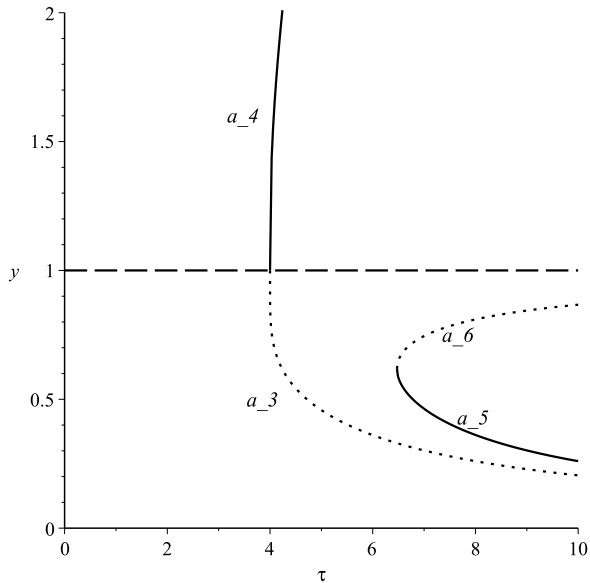}
\caption{ The graphs of the functions $(a_3,a_4)$ and $(a_5,a_6)$ 
	giving 4-periodic boundary laws.}\label{fn2}
\end{figure}

Taking into account the freedom of cyclic permutations of boundary laws we thus proved the following:

\begin{pro}\label{pps} The periodic solutions of the form (\ref{pab}) depend on the parameter $\tau=2\cosh(\beta)$ in the following way. 
\begin{itemize}
\item[1.] If $\tau \leq 4$ then there is a unique solution with $a=b=1$.
\item[2.] If $4<\tau \leq 6$  then there are exactly two solutions with $a=b=1$ and $a=a_3$, $b=b_3$.
\item[3.] If $6<\tau < 2+2\sqrt{5}$  then there are exactly four solutions with $a=b=1$, $a=b=a_1$, $a=b=a_2$ and $(a,b)=(a_3, b_3)$.
\item[4.] If $\tau \geq 2+2\sqrt{5}$  then there are exactly five solutions with $a=b=1$, $a=b=a_1$, $a=b=a_2$, $(a,b)=(a_3, b_3)$ and $(a,b)=(a_5, b_5)$,
\end{itemize}
where the values $a_i$ and $b_i$ are defined above.
\end{pro}

\subsection{Gradient Gibbs measures described by $4$-periodic boundary laws: Identifiability} \label{sec:ident}  In this subsection we will apply the Lemmas \ref{lem: 2-per} and \ref{lem: 4-per} on identifiability to the gradient Gibbs measures which correspond to the periodic solutions given in Proposition \ref{pps}.
Note that a solution described by the parameters $(a,b)$ corresponds to the boundary law
\begin{equation} \label{4-per: specific form}
z_n^{(a^2,b^2)}=\left\{ \begin{array}{lll}
1, \ \ \mbox{if} \ \  n=0 \ \ \mbox{or} \ \ 2 \mod 4,\\[2mm]
a^2, \ \ \mbox{if} \ \ n=1 \mod 4,\\[2mm]
b^2, \ \ \mbox{if} \ \ n=3 \mod 3.
\end{array}
\right.
\end{equation}
We will denote the GGM assigned by Theorem \ref{thm: constrGGM} to a boundary law $z_n^{(a^2,b^2)}$ by $\nu^{(a^2,b^2)}$. \\
%
%

	{\bf Case $4< \tau \leq 6$}: We have $a_3^2+b_3^2 \geq 2(\frac{1}{2}(a_3+b_3))^2=2(\frac{\tau}{\tau-\sqrt{\tau^2-4\tau }}) ^2>2=1^2+1^2$, \\ so $a_3^2+b_3^2 \neq 1^2+1^2$ and $(1^2+1^2)(a_3^2+b_3^2) \neq 4$. Thus $\nu^{(a_3^2,b_3^2)} \neq \nu^{(1,1)} $ by Lemma \ref{lem: 4-per}.\\

	{\bf Case $6 < \tau < 2+2\sqrt{5}$ }:  We have $a_1<1<a_2$ and $a_3^2+b_3^2 > 2(\frac{1}{16}(\tau+\sqrt{\tau^2-4\tau})^2) \geq 2a_2^2$.  Further, as $a_1a_2 \equiv 1$, by Lemma \ref{lem: 2-per} we have that $\nu^{(a_1^2,a_1^2)} \equiv \nu^{(a_2^2,a_2^2)}$. At last $(a_3^2+b_3^2)(a_2^2+a_2^2)>(a_3^2+b_3^2)(a_1^2+a_1^2)>4(a_1a_2)^2=4$. Thus we have {\bf three} different GGMs  associated to boundary laws of the type \eqref{4-per: specific form} via Theorem \ref{thm: constrGGM}.\\

	{\bf Case $\tau \geq 2+2\sqrt{5}$ }: We still have: $a_1^2<1<a_2^2<a_3^2+b_3^2$ and $a_1a_2 \equiv 1$, so again $\nu^{(a_1^2,a_1^2)} \equiv \nu^{(a_2^2,a_2^2)}$. Further $a_5^2+b_5^2 \leq 2b_5^2 < 2$. As $a_1(\tau)$ is monotonically decreasing in $\tau$ and $a_5(\tau)^2+b_5(\tau)^2$
is monotonically increasing in $\tau$ it suffices to numerically calculate $2a_1(2+2\sqrt{5})^6=(\frac{1}{2}(\sqrt{5}-1))^2<0.4<0.76<a_5(2+2\sqrt{5})^2+b_5(2+2\sqrt{5})^2$ to obtain:
	$2a_1^2<a_5^2+b_5^2 <2<2a_2^2<a_3^2+b_3^2$. Thus we have {\bf four} different GGMs  associated to boundary laws of the type \eqref{4-per: specific form} via Theorem \ref{thm: constrGGM}. \\

Hence we have proven the following 	
\begin{thm}\label{tps} For the SOS model (\ref{nu1}) on the binary tree with parameter $\tau=2 \cosh(\beta)$ the following assertions hold
	\begin{itemize}
		\item[1.] If $\tau \leq 4$ then there is precisely one GGM associated to a boundary law of the type \eqref{4-per: specific form} via Theorem \ref{thm: constrGGM}.
		\item[2.] If $4< \tau \leq 6$  then there are precisely two such GGMs.
		\item[3.] If $6<\tau < 2+2\sqrt{5}$  then there are precisely three such GGMs.
		\item[4.] If $\tau \geq  2+2\sqrt{5}$  then there are precisely four such measures.
	\end{itemize}
\end{thm}
\subsection{3-periodic boundary laws on the $k$-regular tree.} To also describe gradient Gibbs measures on the $k$-regular tree for arbitrary $k \geq 2$, we consider a $1$-parameter family of 3-periodic boundary laws which can be examined easily.
Assume $u_n$, $n\in Z$ has the form
\begin{equation}\label{u3p}
u_n=\left\{\begin{array}{ll}
1, \ \ n=0\mod 3\\
a, \ \ n\ne 0\mod 3,\\
\end{array}\right.
\end{equation}
where $a>0$.
Then, by (\ref{Re}) and \eqref{L},  $a$ should satisfy
\begin{equation}\label{e3} 2a^{k+1}-\tau a^k+(\tau-1) a-1=0.
\end{equation}
This equation has the solution $a=1$ independently on the parameters $(\tau, k)$. Dividing both sides by $a-1$
we get
\begin{equation}\label{uy22}
2a^k+(2-\tau)(a^{k-1}+a^{k-2}+\dots+a)+1=0.
\end{equation}
The equation (\ref{uy22}) has again the solution $a=1$ iff $\tau=\tau_0$, where
$$\tau_0:={2k+1\over k-1}.$$

 It is well known (see \cite{Pra}, p.28) that the number of positive
roots of the polynomial (\ref{uy22}) does not exceed the number of sign
changes of its coefficients.
It is obvious that $2-\tau<0$. Thus the number of positive roots of
the polynomial (\ref{uy22}) is at most 2.

The following lemma gives the full analysis of the equation (\ref{uy22}):

\begin{lemma}\label{l6} For each $k\geq 2$,
there is exactly one critical value of $\tau=2\cosh(\beta)$, called $\tau_c=\tau_c(k)$, such that
\begin{itemize}
\item[1.] $\tau_c<\tau_0$;
\item[2.]  if $\tau<\tau_c$ then (\ref{uy22}) has no positive solution;
\item[3.] if $\tau=\tau_c$ then the equation has a unique positive solution;
\item[4.] if $\tau>\tau_c$, $\tau\ne \tau_0$
then it has exactly two solutions;
\item[5.] if $\tau=\tau_0$, then the equation has two solutions, one of which is $a=1$.
\end{itemize}
\end{lemma}
\begin{proof}
Solving (\ref{uy22}) with respect to $\tau$ we get
$$\tau=\psi_k(a):=2+{2a^k+1\over a^{k-1}+a^{k-2}+\dots+a}. $$
We have $\psi_k(a)>2$, $a>0$ and $\psi_k'(a)=0$ is equivalent to
\begin{equation}\label{ho}
2\sum_{j=1}^{k-1}(k-j)a^{k+j-1}-\sum_{j=1}^{k-1}ja^{j-1}=0.
\end{equation}
The last polynomial equation has exactly one positive solution,
because signs of its coefficients changed only one time,
and at $a=0$ it is negative, i.e. -1 and at $a=+\infty$ it is positive.
Denote this unique solution by $a^*$. Then $\psi_k(a)$ has unique minimum
at $a=a^*$, and $\lim_{a\to 0}\psi_k(a)=\lim_{a\to +\infty}\psi_k(a)=+\infty$ (see Fig.\ref{fn3}).
\begin{figure}
\includegraphics[width=8cm]{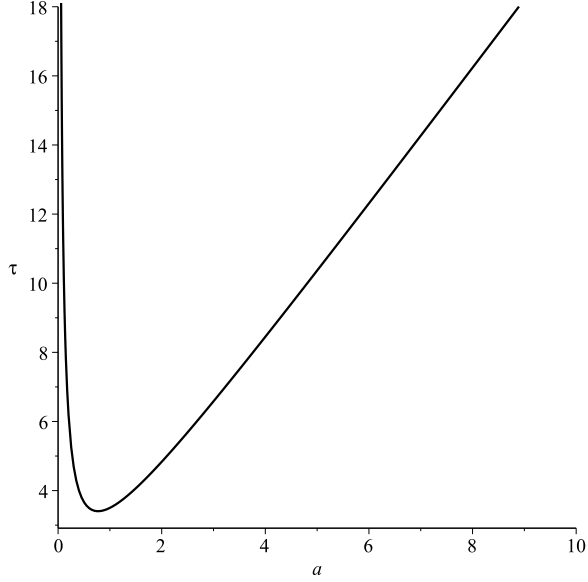}
\caption{ The graphs of the function $\psi_3(a)$.}\label{fn3}
\end{figure}
Thus
$$\tau_c=\tau_c(k)=\min_{a>0}\psi_k(a)=\psi_k(a^*).$$
Note that $a^*\ne 1$, i.e. $a=1$ does not satisfy (\ref{ho}).
Therefore $$\tau_0=\psi_k(1)>\tau_c=\min_{a>0}\psi_k(a)=\psi_k(a^*).$$
These properties of $\psi_k(a)$ completes the proof.
\end{proof}
Thus taking into account Lemma \ref{lem: 3-per} and Lemma \ref{l6} we obtain the following:
\begin{thm} For the SOS-model on the $k$-regular tree, $k \geq 2$, with parameter $\tau=2\cosh(\beta)$ there are numbers $0<\tau_c<\tau_0$ such that the following holds: 
\begin{itemize}
	\item[1.] If $\tau<\tau_c$ then there are no GGM corresponding to 
	nontrivial $3$-periodic boundary laws of the type \eqref{u3p} via Theorem \ref{thm: constrGGM}. 
	\item[2.] At $\tau=\tau_c$ there is a unique GGM corresponding to a
	nontrivial $3$-periodic boundary law of the type \eqref{u3p} via Theorem \ref{thm: constrGGM}.
	\item[3.] For $\tau>\tau_c$, $\tau\ne \tau_0$ (resp. $\tau=\tau_0$) there are exactly
	two such (resp. one) GGMs. 
\end{itemize}	
The GGMs described above are all different from the GGMs mentioned in Theorem \ref{tps}.
\end{thm}
\begin{rk} In case $k=2$ one easily finds $\tau_c(2)=2(1+\sqrt{2})\approx 4.83$ and $\tau_0=5$.
Two positive solutions are (see Fig. \ref{fn1}):
$$a=b_{1,2}:={1\over 4}(\tau-2\pm\sqrt{(\tau-2)^2-8})$$
\end{rk}
\begin{rk}
It was shown in Section 5 of \cite{KS} that the equations for $3$-periodic boundary laws of the general form 
\begin{equation}
z_i=\left\{\begin{array}{ll}
1, \ \ i=0\mod 3\\
a, \ \ i= 1\mod 3\\
b, \ \ i=2 \mod 3,\\ 
\end{array}\right.
\end{equation} 
can be identified with the boundary law equations of a Potts model on the same regular tree at a different effective inverse temperature. An explicit discussion of the transition temperature was given only on the binary tree. This correspondence however explains that all $3$-periodic boundary law solutions are of the type \eqref{u3p}.
\end{rk}

 \section*{ Acknowledgements}

 U.A. Rozikov thanks the  RTG 2131, the research training
group on High-dimensional Phenomena in Probability - Fluctuations and Discontinuity,
 and the Ruhr-University Bochum (Germany) and  University Paris Est Cr\'eteil (France), for financial support and hospitality. 
 
 F. Henning thanks the RTG 2131.
 
 The research of A. Le Ny has been partly funded by the B\'ezout Labex, funded by ANR, reference ANR-10-LABX-58.

\end{document}